\documentclass[11pt]{amsart}

\usepackage{latexsym,enumerate}

\usepackage{amsmath,amssymb,amsthm,amsfonts,latexsym}
\usepackage{pstricks, pst-node, pst-text, pst-3d}
\usepackage[bookmarks]{hyperref}

\let\=\partial

\newtheorem {pro}{Proposition}[section]
\newtheorem {thm}[pro]{Theorem}
\newtheorem{lem}[pro]{Lemma}

\theoremstyle{definition}
 \newtheorem {rem}[pro]{Remark}
\newtheorem {dfn}[pro]{Definition}

\newtheorem {ste}{Step}
\newtheorem {step}{Step}

\newcommand{\G}{\mathbb{G}}
\newcommand{\R}{\mathbb{R}}

\newcommand{\N}{\mathbb{N}}
\newcommand{\C}{\mathcal{C}}

\newcommand{\et}{\quad \mbox{and} \quad }

\newcommand{\St}{\mathcal{D}}

\newcommand{\spl}{\mathcal{D}^+}
\newcommand{\spn}{\mathcal{D}^+(\R^n)}

\newcommand{\don}{\mathcal{D}^0(\R^n)}
\newcommand{\din}{\mathcal{D}^\infty(\R^n)}
\newcommand{\drn}{\mathcal{D}(\R^n)}

\newcommand{\xib}{\overline{\xi}}

\newcommand{\xt}{\tilde{x}}

\newcommand{\U}{\mathcal{U}}
\newcommand{\V}{\mathcal{V}}
\newcommand{\W}{\mathcal{W}}

\newcommand{\D}{\mathcal{D}}

\newcommand{\ep}{\varepsilon}
\newcommand{\dej}{{\delta_j}}

\newcommand{\pa}{\partial}

\newcommand{\om}{\infty}

\title {Efroymson's approximation Theorem for globally subanalytic functions}



\author[A. Valette  and G. Valette]{Anna Valette  and Guillaume Valette}

\address[A. Valette]{Instytut Matematyki Uniwersytetu
Jagiello\'nskiego, ul. S \L ojasiewicza, Krak\'ow, Poland}
\email{anna.valette@im.uj.edu.pl}\address[G. Valette]{Instytut Matematyki Uniwersytetu
Jagiello\'nskiego, ul. S \L ojasiewicza, Krak\'ow, Poland}\email{guillaume.valette@im.uj.edu.pl}

\keywords{subanalytic functions, o-minimal structures, analytic functions, Nash functions, approximation theorem, Efroymson's theorem}

\thanks{Research partially supported by the NCN grant  2014/13/B/ST1/00543.}

\subjclass[2010]{14P10,  32B20, 41A30, 58A07}

\begin{document}\maketitle
\begin{abstract}
Efroymson's Approximation Theorem asserts that if $f$ is a $\C^0$ semialgebraic mapping on a $\C^\om$ semialgebraic submanifold $M$ of $\R^n$ and if $\ep:M\to \R$ is a positive continuous semialgebraic function then there is a $\C^\infty$ semialgebraic function $g:M\to \R$ such that $|f-g|<\ep$. We prove a generalization of this result to the globally subanalytic category. Our theorem actually holds in a larger framework since it applies to every function which is definable in a polynomially bounded o-minimal structure (expanding the real field) that admits $\C^\infty$ cell decomposition. We also establish approximation theorems for Lipschitz  and $\C^1$ definable functions.
\end{abstract}

\section{Introduction} 
 In \cite{ef_approx}, G. Efroymson proved an approximation theorem for continuous semialgebraic functions (see also \cite{bcr,shiota,shiota2}). This result can be stated as follows:
 \begin{thm}\label{thm_ef_original}
  Let $M$ be a Nash submanifold of $\R^n$ and let $\ep:M\to \R$ be a positive continuous semialgebraic function. Given a continuous  semialgebraic function $f$ on $M$ there is a Nash function $g$ on $M$ such that $|f(x)-g(x)|<\ep(x)$, for all $x\in M$.  
 \end{thm}
Let us recall that a Nash function is a $\C^\om$ semialgebraic function. 
The rigidity of this class of  functions makes this result very attractive. 
 Shiota, who gave an independent proof of this result \cite{shiota,shiota2}, also achieved a stronger theorem ensuring that, when $f$ is $\C^m$,   it is possible to approximate the $m$ first derivatives as well.

The aim of the present article is to generalize Efroymson's theorem to the globally subanalytic category (Theorem \ref{thm_ef_C_0} below).
Our framework  is however much bigger than this category since our approximation theorems hold every polynomially bounded o-minimal structure expanding the real field that admits $\C^\om$ cell decomposition. In particular, it applies to quasi-analytic Denjoy-Carleman classes \cite{rsw}. These assumptions can be considered as weak since,  although there exist examples of o-minimal structures that do not admit $\C^\om$ cell decomposition \cite{celldec}, such examples are rather difficult to construct and all the basic examples of o-minimal structures do possess this property. 

We assume that the structure is polynomially bounded rather for convenience since it is known that Efroymson's theorem holds in the non polynomially bounded case \cite{fisher_exp}. Indeed, in that case, the exponential function is necessarily in the structure \cite{exp} and one can construct $\C^\om$ definable partitions of unity.

The difficulty to prove Efroymson's theorem in polynomially bounded structures is that the rings of $\C^\om$ definable functions are quasi-analytic, in the sense that the Taylor series of such a function-germ cannot vanish, unless this function-germ is identically zero. Basic techniques of approximation theory, such as partitions of unity, therefore have to be excluded. On the other hand, the known proofs of Theorem \ref{thm_ef_original} heavily rely on the algebraicity of the semialgebraic functions. This is the reason why we adopted a different approach, proving an approximation theorem for definable manifolds (Proposition \ref{lem_plongement}). 

 We should nevertheless emphasize that our argument does not provide  a new proof of Theorem \ref{thm_ef_original} for we make use of this result.  We will also provide a  approximation theorem for $\C^1$ functions (Theorem \ref{thm_ef_C_1}) that can be seen as a generalized version of Shiota's result (see \cite{shiota, shiota2}) and an approximation theorem for Lipschitz functions (Theorem \ref{thm_approx_lips}) with some uniform bounds for the Lipschitz constant of the approximations. This improves the results that were obtained in \cite{escribano,fisher_lips}. We wish to stress the fact that, if 
Theorems \ref{thm_approx_lips} and  \ref{thm_ef_C_1} are of their own interest, they are also definitely needed even if one is only interested in constructing $\C^0$ approximations. The reason is that we shall need to prove an approximation theorem for manifolds, and this kind of result is less difficult to achieve in the $\C^1$ category (see section \ref{sect_plongement}). 

The aim of this article being the generalization of Efroymson's approximation theorem, we only focused on approximations of $\C^0$ and $\C^1$ functions. A systematic generalization of the method (to the order $m$) would however provide a theorem for $\C^m$ functions with the approximation of the $m$ first derivatives, which would then be the subanalytic counterpart of Shiota's theorem \cite{shiota, shiota2}. The cases $m=0$ and $m=1$ are nevertheless satisfying for most of the applications of Efroymson's theorem. In particular, a byproduct of the main theorem of this article is that Mostowski's separation theorem \cite{mostowski, bcr} holds in the subanalytic category. Embedding theorems for $\C^\om$ manifolds or applications to $\C^\om$ triviality of mappings could also be derived.
\begin{section}{Framework and basic facts}
Throughout this article, $n$ and $k$ stand for two integers. Given two functions $f$ and $g$, we write $f<g$  if $f(x)<g(x)$ everywhere both functions are defined. Given a subset $A$ of $\R^n$ on which both $f$ and $g$ are defined, we write ``$f<g$ on $A$'', if  this inequality holds for all $x\in A$.  

Given a function $\xi:A\to \R$, with $A\subset \R^n$, $\Gamma_\xi$ will stand for the graph of $\xi$. If $\xi':A\to \R$ is another function, we set
$$[\xi,\xi']:=\{(x,y)\in A \times \R:\xi(x)\le y\le \xi'(x)\}.$$
The sets $(-\xi,\xi']$ as well as $[-\xi,\xi')$ and $(-\xi,\xi')$ are defined analogously. We also set
$$[\xi,+\infty):=\{(x,y)\in A\times \R:\xi(x)\le y\}$$
and define $(-\infty,\xi]$ as well as $(-\infty,\xi)$ and $(\xi,+\infty)$ analogously.

Given $x\in \R^n$, we denote by $|x|$ the Euclidean norm of $x$ and by $d(x,A)$ the  (Euclidean) distance of $x$ from the subset $A$.    If  $\delta \in \spl(A)$, we can then define a neighborhood of $A$ by setting:
$$\V_{\delta}(A):=\{x\in \R^n: d(x,A)< \delta(x) \}.$$

Unless otherwise specified, by ``manifold'' we will mean ``manifold without boundary''.  Manifolds (without boundary) can however also be considered as manifolds with boundary (which is then empty).
   \subsection{O-minimal structures.}  
   A \textit{structure} (expanding $(\mathbb{R}, + ,.)$) is a family $\mathcal{D} = (\D_n)_{n \in \mathbb{N}}$ such that for each $n$ the following properties hold
\begin{enumerate}
\item [(1)] $\D_n$ is a Boolean algebra of subsets of $\mathbb{R}^n$.
\item[(2)] If $A \in \D_n$ and $B\in \D_k$ then $ A\times B$  belongs to $\D_{n+k}$.
\item[(3)] $\D_n$ contains $\{x \in \mathbb{R}^n: P(x) = 0\}$, where $P \in \mathbb{R}[X_1,\ldots, X_n]$.
\item[(4)] If $A \in \D_n$ then $\pi(A)$ belongs to $\D_{n-1}$, where $\pi: \mathbb{R}^n \to \mathbb{R}^{n-1}$ is the standard projection onto the first $(n-1)$ coordinates.
 \end{enumerate}
 Such a family $\mathcal{D}$ is said to be \textit{o-minimal} if in addition:
 \begin{enumerate}
 \item [(5)] Any set $A\in \D_1$ is a finite union of intervals and points.
 \end{enumerate}
A set belonging to the structure $\mathcal{D}$ is called a {\it definable set} and a map whose graph is in the structure $\mathcal{D}$ is called  a {\it definable map}.

A structure $\mathcal{D}$ is said to be \textit{polynomially bounded} if for each $\mathcal{D}$-function $f: \mathbb{R}\to \mathbb{R}$, there exists a positive number $a$ and  $p \in \mathbb{N}$ such that $|f(x)| < x^p$ for all $x > a$.

Examples of polynomially bounded o-minimal structures are the semi-algebraic sets, the globally subanalytic sets \cite{dd,lr} but also the so called $x^\lambda$-sets \cite{xlambda,lr} as well as the structures defined by the quasi-analytic Denjoy-Carleman classes of functions \cite{rsw}. We refer to \cite{coste_bleu,dries_omin} for basic facts about o-minimal structures.

We say that an o-minimal structure {\it admits $\C^\om$ cell decomposition} if for each definable function $f:U\to \R $, $U$ open subset of $\R^n$, there is a definable open dense subset of $V$ of $U$ on which $f$ is $\C^\om$.

All the just above mentioned examples of polynomially bounded o-minimal structures  admit  $\C^\om$ cell decomposition. Our theorems will therefore apply to these structures.

Throughout this article, $\mathcal{D} = (\D_n)_{n \in \mathbb{N}}$  will stand for a fixed polynomially o-minimal structure (expanding $(\mathbb{R}, + ,.)$)  admitting $\C^\om$ cell decomposition. The term definable will refer to this structure. Since {\it all the sets and mappings will be definable}, we will often omit to mention it in the proofs. For the sake of completeness of statements, we however have chosen to emphasize it in the definitions, propositions, and theorems.

Given  $A\in \D_n$ and $B\in \D_k$, we denote by $\D(A,B)$ the set of  definable  mappings $\xi:A\to B$ and by $\D(A)$ the set of  definable functions $\xi:A\to \R$, whereas  $\spl(A)$ will stand for the set of positive definable continuous functions on the set $A$. 

Given $m\in \N\cup \{\infty\}$, we write $\D^m(A)$ (resp. $\D^m(A,B)$) for the elements of $\D(A,B)$  (resp. $\D(A)$) that are $\C^m$ (a function will be said to be $\C^m$ on  $A$ if it extends to a $\C^m$ function on an open neighborhood of $A$ in $\R^n$).

\begin{subsection}{Efroymson's  topology.} 
Given  a mapping $\xi:A\to B$, with $A\in \D_n$, $B\in \D_k$, and $\ep \in \spl(A)$ we set: $$\U_\ep^0(\xi)=\{\zeta\in \St^0(A,B):  |\xi-\zeta|<\ep\}.$$
The {\it $\C^0$ Efroymson topology} on $\D^0(M,N)$ is the topology 
 for which a basis of neighborhoods of $f$ is given by the family $(\U_\ep^0(f))_{\ep \in \spl(A)}$. Efroymson's theorem (Theorem \ref{thm_ef_original}) yields that $\D^\om(A)$ is dense in $\D^0(A)$ in the semialgebraic category (for each definable set $A$).

Let now $M\subset \R^n$ be a $\C^1$ definable submanifold (possibly with boundary) and let $f\in \D^1(M,N)$, where $N$ is a $\C^1$ definable submanifold of $\R^k$ (possibly with boundary).  We will write $|d_xf|$ for the norm of $d_xf$ (as a linear mapping) derived from the Euclidean norm $|.|$.   We will write $|df|$ for the function defined by $x\mapsto |d_xf|$. We then set $$|f|_1:=|f|+|df|,$$ 
 and, given $\ep \in \spl(A)$,$$\U_\ep^1(f)=\{g\in \St^1(M,N):  |f-g|_1<\ep\}.$$
The {\it $\C^1$ Efroymson topology} on $\D^1(M,N)$ is the topology 
 for which a basis of neighborhoods of $f$ is given by the family $(\U_\ep^1(f))_{\ep \in \spl(M)}$.  It was proved by M. Shiota that 
 $\D^\infty(M,N)$ is dense in $\D^1(M,N)$ in the semialgebraic category  \cite{shiota, shiota2} whereas Escribano showed that $\D^m(M,N)$ is dense in $\D^l(M,N)$  for  $0\le l\le m<\infty $ \cite{escribano}. 


Let for $j\in \N$ positive and $x\in \R^n$ \begin{equation}\label{eq_delta_j}
\delta_j (x)    := \frac{1}{(1+j^2)(1+|x|^{2j})}.
    \end{equation}   Since the $\delta_j$ are Nash functions, they all belong to $\St^\infty(\R^n)$.
    
Remark also that,  since the structure is assumed to be polynomially bounded, by \L ojasiewicz inequality, the sequences $\dej$ and $ |d\delta_j|$   converge to zero (in $\D^0(\R^n)$). In particular, for every $\eta\in \spn$ there is $j$ such that $\eta\ge \delta_j$. As a matter of fact, a function $f:M\to N$ lies in the closure of a set $Z\subset \D^1(M,N)$ if and only if there is a sequence $g_j\in Z$ such that $|f-g_j|_1<\dej$, for all $j$. Of course, the analogous fact can be observed for the $\C^0$ Efroymson topology.

 We now state our approximation theorem:

\begin{thm}\label{thm_ef_C_0}
For every $\C^\infty$ definable submanifold $M$ of $\R^n$,  $\D^\om(M)$ is dense in $\D^0(M)$ for the $\C^0$ Efroymson topology.
\end{thm}

   In other words, given any $f\in \D^0(M)$ and $\ep \in \spl(M)$, there is $g\in\D^\om(M)$ such that $|f-g|<\ep$.  As $\D^1(M)$ is dense in $\D^0(M)$, this theorem follows from Theorem \ref{thm_ef_C_1}, which asserts that $\D^\om(M)$ is dense in $\D^1(M)$ (for the $\C^1$ Efroymson topology).

%
%
%
%
%
%
%
%
%
\end{subsection}

 \begin{subsection}{Lipschitz functions.}  A function $\xi:A\to \R$ is {\it Lipschitz} if it is Lipschitz with respect to the metric induced by $\R^n$, i.e., if there is a constant $L$ such that $|\xi(x)-\xi(x')|\le L|x-x'|$. The smallest such constant is then called the {\it Lipschitz constant of $\xi$} and is denoted $L_\xi$. Every Lipschitz function $\xi\in \D(A)$, $A\subset \R^n$, may be extended to an $L_\xi$-Lipschitz function  $\xib\in \D(\R^n)$ defined as
 $$\xib(x):=\inf \{\xi(y)+L_\xi|x-y|: y\in A\} .$$ 
 
 A straightforward computation yields the following fact that we will use all along this article: if $\xi$ is Lipschitz we have for $x=(\xt,x_n)\in \R^{n-1}\times \R$ \begin{equation}\label{eq_dist_au_graph}
                                                    |x_n-\xi(\xt)
  |\le (L_\xi+1)d(x,\Gamma_\xi).
                                                   \end{equation}

                                                    The following related fact will be of service.
                                                    \begin{lem}\label{lem_v_w}
 Given a positive constant $L$, there are constants $0<c<\kappa<1$ such that for all $L$-Lipschitz functions $\delta\in\spn$ and  $\xi\in\drn $ we have:
 $$  \V_{c \delta}(\Gamma_{\xi})\;\subset\;[\xi-\kappa\sigma,\xi+\kappa\sigma]\;\subset \;\V_{ \delta}(\Gamma_{\xi}),$$
 where for $x \in \R^{n-1}$ we have set    $\sigma(x):=\delta(x,\xi(x))$.  
\end{lem}
\begin{proof}
We first check that the second inclusion holds for $\kappa$ small enough and then we will check that the first one holds for $c$ small enough (with $\kappa$ fixed). Observe that if $x=(\xt,x_n)\in \R^{n-1}\times \R$ satisfies $|\xi(\xt)-x_n|\le \kappa\sigma(\xt)$ then $\delta(\xt,\xi(\xt))-\delta(x)\le L\kappa\sigma(\xt)$ which implies  for $\kappa\le \frac{1}{2L}$    that $\sigma(\xt)\le 2\delta(x)$ so that for $\kappa< \frac{1}{2}$ $$d(x,\Gamma_{\xi})\le |\xi(\xt)-x_n|\le \kappa\sigma(\xt)\le 2\kappa\delta(x)<\delta(x),$$
yielding the right-hand-side inclusion.
 Let us now fix a $\kappa$ sufficiently small and show the left-hand-side inclusion.
 If $x=(\xt,x_n)\in \R^n$ satisfies $d(x,\Gamma_{\xib})<c \cdot \delta(x)$ for some $c$ then, by (\ref{eq_dist_au_graph})
\begin{equation}\label{eq_pr_dej_dejp}
 |x_n-\xi(\xt)|\le (L_{\xi}+1) c \cdot \delta(x)
\end{equation}
     so that for $c$ small enough
$$|\delta(x)-\delta(\xt,\xi(\xt))|\le |x_n-\xi(\xt)|\le (L_{\xi}+1) c\cdot \delta(x) <\frac{1}{2}\delta(x), $$
 which implies that $\delta(x)\le 2\sigma(\xt)$ for such $x=(\xt,x_n)$, which, by (\ref{eq_pr_dej_dejp}), entails in turn that for $c$ small enough
 $|x_n-\xi(\xt)|<\kappa \sigma(\xt). $
\end{proof}

 \begin{dfn}\label{dfn_lipschitz_cells}
 We are going to define the Lipschitz cells of $\R^n$, which requires to first define the cells of $\R^n$ inductively on $n$. Every  subset of $\R^0=\{0\}$  is a cell. A definable subset $C$ is a cell of $\R^{n}$  if there is a cell $D$ of $\R^{n-1}$ such that one of the following conditions holds:
  \begin{enumerate}
   \item\label{item_type_1}  $C=\Gamma_{\xi}$ with  $\xi\in \D^\infty(D)$.
   \item\label{item_type_2}  $C=(\xi,\xi')$, where $\xi$ is either equal to $-\infty$ or  a $\C^\om$   function on $D$, and  $\xi'$ is either equal $+\infty$ or a $\C^\om$   function  on $D$ satisfying $\xi< \xi'$.
  \end{enumerate}
We then call $D$ the {\it basis of $C$}. Now, a cell $C$ is a {\it Lipschitz cell} if its basis $D$ is a Lipschitz cell and if in addition the functions appearing in its description (in $(\ref{item_type_1} )$ or $(\ref{item_type_2} )$ above) are Lipschitz.  
A {\it Lipschitz open cell of $\R^n$} is a Lipschitz cell of $\R^n$ of dimension $n$.

 %

 If $C$ is a Lipschitz open cell of $\R^n$ and $\eta \in \spn$, we set:
 \begin{equation}\label{eq_w_eta}
  \W_{\eta}(C):=C\setminus \V_\eta(\R^n\setminus C) .\end{equation}
 \end{dfn}

 The following result is a well-known fact about definable sets which relies on the assumption of existence of $\C^\infty$ cell decompositions.  The stratification provided by this proposition will be useful in the proof of Theorem \ref{thm_approx_lips} and Proposition \ref{pro_rn}. Let us recall that a {\it stratification}  of a set is a finite partition of this set into $\C^\om$ manifolds, called {\it strata}, 
such that the closure of one stratum is the union of some strata.
\begin{pro}\label{pro_stratification}
 Given $f\in \D(\R^n)$, there is a stratification of $\R^n$ such that $f$ is $\C^\infty$ on every stratum and such that each stratum is a Lipschitz cell after a possible orthonormal change of coordinates.
\end{pro}
\end{subsection}

 \begin{subsection}{Pullback and pushforward.} Given two  mappings $h:\R^p\to \R^n$ and  $u:U\to  \R^k$, with $U\subset \R^n$, we denote by $h^*u:h^{-1}(U) \to  \R^k$ the mapping induced by pull-back, i.e., the mapping defined by $h^*u(x):=u(h(x))$.
 
 In the case where $h$ is a homeomorphism (onto its image) and $v:V\to \R^k$ is a mapping with $V\subset \R^p$ we denote by $h_*v$ the push-forward of $v$, which is merely the pull-back of $v$ under $h^{-1}$. 

%

\end{subsection}

\end{section}
\begin{section}{Some preliminary lemmas}
 The lemmas listed in this section provide tools to construct definable $\C^\om$ functions. They rely on Efroymson's approximation theorem and elementary facts.    We recall that a $\C^\infty$ semialgebraic function is called {\it a Nash function}.
\begin{lem}\label{lem_psiep1}
There is a Nash function $\psi:(0,1)^2\times \R \to [0,3)$ satisfying  
\begin{enumerate}
 \item\label{item_psi_ge_1} $\psi(\mu,\delta,y)>1$ whenever  $(\mu,\delta,y)\in(0,1)^2\times (-\infty,0] $. 
 \item\label{item_mu} $\psi(\mu,\delta,y) <\mu $ and $|d\psi(\mu,\delta,y)| <\mu $ whenever  $y \in [\delta,+\infty)$. 
 \item There is a constant $A$ such that for all  $(y,\delta,\mu)$ we have 
 \begin{equation}\label{eq_der_psi}|d\psi(\mu,\delta,y)|\leq A\; \mbox{ if } \;y\le 0 \et|d\psi(\mu,\delta,y)|\leq \frac{A}{\delta} \;\mbox{ if }\; y\ge 0.\end{equation}
\end{enumerate}
\end{lem}
\begin{proof}
 Let $\theta:\R \to [0,2]$ be a $\C^1$ piecewise polynomial function satisfying $\theta(x)=2$ if $x\le 0$ and $\theta(x)=0$ if $x\ge 1$. For $(\mu,\delta,y)\in (0,1)^2\times \R$ set $$u(\mu.\delta,y):=\theta(\frac{y}{\delta})+\frac{\mu^2}{4}. $$ 
 We first check that $|du(\mu,\delta,y)|\le \frac{A}{\delta}$, for some constant $A$ independent of $(\mu,\delta,y)$. Indeed, a simple computation of partial derivative yields such an estimate for $\frac{\pa u}{\pa \mu} $ and $\frac{\pa u}{\pa y}$. Moreover,
 if $y \notin [0,\delta]$ then $\theta'(\frac{y}{\delta})=0$ which implies that $\frac{\pa u}{\pa \delta}
 (\mu,\delta,y)=0$ (for all $\mu$). We thus can assume $y \in [0,\delta]$, which entails $\frac{y}{\delta^2}\le \frac{1}{\delta}$, and we can conclude that there is a constant $A$ such that  $$|\frac{\pa u}{\pa \delta}| =\frac{|y|}{\delta^2}\cdot |\theta'(\frac{y}{\delta})|\le \frac{A}{\delta} .$$ 
 
By Efroymson's Theorem, we can find a Nash function $\psi$ on $(0,1)^2\times \R$ such that $|u(\mu,\delta,y)-\psi(\mu,\delta,y)|_1<\frac{\mu}{2}$. This function has all the required properties. 
\end{proof}
%

\begin{lem}\label{lem_phiep1}
There is a Lipschitz Nash function $\Phi:(0,1)\times  \R\to \R$ such that 
\begin{enumerate}
\item  For all $\delta\in \R $ we have $\Phi(\delta,y)\in (0,2\delta)$, for all $y\in (-\infty,\delta]$.
\item On the set  $\{ (\delta,y)\in(0,1)\times \R  : y\in [\delta,+\infty)\}$, we have
\begin{equation}\label{eq_phi}
 |(\Phi- \Lambda)(\delta,y)| <\delta \et |d(\Phi- \Lambda)(\delta,y)| <\delta
\end{equation}
 where $\Lambda:(0,1)\times  \R\to \R$ is the function defined by $\Lambda(\delta,y)=y$.
 \end{enumerate}  
\end{lem}
\begin{proof}
Let $\theta:\R\to [0,1]$ be a piecewise polynomial $\C^1$ function satisfying $\theta(y)=1$ if $y<0$, $\theta(y)=0$ if $y\ge 1$.  
Define first a $\C^1$ one-variable function on $(0,1)\times  \R$ by 
$$\phi(\delta,y)=\theta(\frac{y}{\delta})\cdot \frac{\delta}{1+y^2} +(1-\theta(\frac{y}{\delta}))y.$$  
The desired function will be provided by an approximation of $\phi$. We first check that $\phi$ is a Lipschitz function.  It suffices to show that $|d\phi|$ is bounded.  Rewriting $\phi$ as 
$$\phi(\delta,y)=\theta(\frac{y}{\delta})\cdot(\frac{\delta}{1+y^2}-y) +y,$$
we see that it is enough to prove that the function $\theta(\frac{y}{\delta})\cdot(\frac{\delta}{1+y^2}-y)$ has bounded derivative.  If $y<0$ or if   $y>\delta$ then $\theta'(\frac{y}{\delta})$ vanishes and it is easy to find a bound for the derivative (it is clear that $(\frac{\delta}{1+y^2}-y)$ is a Lipschitz function).  We thus will focus on the couples $(\delta,y)\in (0,1)\times \R$ satisfying $y \in [0,\delta]$. For such $(\delta,y)$, a straightforward computation yields that the norm of the derivative of the function $\theta(\frac{y}{\delta})$ is not greater than $\frac{A}{\delta}$, for some constant $A$, and since $|\frac{\delta}{1+y^2}-y|\le 2\delta$, we see that the desired bound is easy to find.

We now claim   that $\phi(\delta,y) \in [0,\delta]$,  for all $(\delta,y)$ such that $y\le \delta$. Indeed, if $y<0$ then $\theta(\frac{y}{\delta})=1$, so that  $\phi(\delta,y)=\frac{\delta}{1+y^2}$, which clearly belongs to the interval $[0,\delta]$.
 If $y\in [0,\delta]$ then, because both $y$ and $\frac{\delta}{1+y^2}$ belong to this interval, it is clear from the definition of $\phi$ that  $\phi(\delta,y)$ sits in this interval too, as claimed. 

By Efroymson's Theorem, we know that there is a $\C^\om$ function $\phi':(0,1)\times \R\to \R $ such that $|\phi'-\phi|_1<\ep$, on $(0,1)\times \R$, where $\ep(\delta,y):=\frac{\delta^2}{4}$. Set then $$\Phi(\delta,y):=\phi'(\delta,y)+\frac{\delta^2}{4}.$$ 
Since $\phi$ coincides with $\Lambda$ on the set $\{ (\delta,y)\in (0,1)\times \R  : y\in [\delta,+\infty)\}$, we see that $(2)$ clearly holds.  Moreover, if  $y\in (-\infty,\delta]$ then, because $\phi(\delta,y)\in [0,\delta]$, we must have $\phi'(\delta,y)\in [-\frac{\delta^2}{4},\delta+\frac{\delta^2}{4}]$, which entails that $\Phi(\delta,y)\in (0,2\delta)$, yielding $(1)$.
%
\end{proof}

\begin{lem}\label{lem_phiep}
There is a Lipschitz Nash function $\Psi:(0,1) \times \R \to (0,1)$ such that 
on the set  $\{ (\delta,y)\in (0,1)\times \R: y\in (\delta,1-\delta) \}$, we have
\begin{equation}\label{eq_psi}
 |(\Psi- \Lambda)(\delta,y)| <\delta \et |d(\Psi- \Lambda)(\delta,y)| <\delta
\end{equation}
 where $\Lambda:(0,1)\times  \R\to \R$ is the function defined by $\Lambda(\delta,y)=y$.
\end{lem}
\begin{proof}
 Let $\Phi:(0,1)\times \R\to \R$ be the Lipschitz $\C^\om$ mapping provided by Lemma \ref{lem_phiep1} and set $\tilde{\Phi}(\delta,y):=\Phi(\delta,y)+\delta^2$. Fix a  positive constant $L$ and, for $(\delta,y)\in (0,1)\times \R$, set
 $$\Psi(\delta,y):= 1-\Phi(\frac{\delta^2}{2L^2},1-\tilde{\Phi}(\frac{\delta}{L},y)), $$
and let us check that if $L$ is chosen large enough then this mapping has the required properties.

Observe that $(1)$ and $(2)$  of Lemma \ref{lem_phiep1} imply that   \begin{equation}\label{eq_borne_phi_c}
                                                                    \Phi(\delta,y)< \max (y,0)+2\delta.
                                                                   \end{equation}
             Furthermore, since $\Phi(\delta,y)$ is positive for all $(\delta,y)$, we have:  \begin{equation}\label{eq_u}
                                    u:=\tilde{\Phi}(\frac{\delta}{L},y)= \Phi(\frac{\delta}{L},y)+\frac{\delta^2}{L^2} \ge \frac{\delta^2}{L^2}
                                   \end{equation}
We deduce that:
                                                                  $$ \Psi(\delta,y)=1-\Phi(\frac{\delta^2}{2L^2}, 1-u)\overset{(\ref{eq_borne_phi_c})}{>} \min(u,1)-\frac{\delta^2}{L^2}\overset{(\ref{eq_u})}{\ge} 0, $$
                                                                  for $L\ge 1$. Moreover, since $\Phi(\delta,y)$ is positive for all $(\delta,y)$, we clearly have
    $$ \Psi(\delta,y)=1-\Phi(\frac{\delta^2}{2L^2}, 1-u) <1,$$
   yielding that $\Psi(\delta,y)\in (0,1)$. A straightforward computation of derivative then yields that if $L$ is chosen sufficiently large then $(2)$ of Lemma \ref{lem_phiep1} implies (\ref{eq_psi}).
 \end{proof}

%
%

 We shall also need the following elementary fact.
 
\begin{lem}\label{lem delta ep}
 Let $\xi \in \St^0(U)$, where $U\in \D_n$ is a open set, and let $\ep\in \spl(U)$. There is $\delta \in \spl(U)$ such that for all $x$ and $x'$ in $U$, $|x-x'|< \delta(x)$ entails $|\xi(x)-\xi(x')|< \ep(x)$. 
\end{lem}
\begin{proof}
 For  $x\in U$, let $\nu(x):=\sup \{\nu \in (0,\infty): \forall \, x' \in U, \;|x-x'|<\nu \implies |\xi(x)-\xi(x')|< \ep(x) \}$. This function is definable and bounded below away from zero on compact subsets of $U$. Hence, there must be $\delta \in \spl(U)$ such that $\delta(x)<\nu(x)$.
\end{proof}

\end{section}

\begin{section}{Approximations of Lipschitz functions}
In this section we establish that the definable Lipschitz $\C^\infty$ functions are dense in the set of definable Lipschitz functions. We start with the following technical lemma.
\begin{lem}\label{lem_app_open_cell}
 Let $C$ be a Lipschitz open cell of $\R^n$ and let $f\in \D^\infty(C) $. For each $\ep$ and $\eta$ in $\spn$, there is    $g\in \din$, such that $|f-g|_1<\ep$ on $\W_\eta(C)$. 
\end{lem}
\begin{proof}
 Fix two functions $\ep$ and $\eta$ in $\spn$.  Let $D_i$ be the image of $C$ under the orthogonal projection  onto $\R^i$ (so that $D_n=C$).

We are going to construct for each $i\le n$ (by downward induction on $i$) a $\C^\om$ function $g_i:D_i \times \R^{n-i}\to \R$ which satisfies $|g_{i+1}-g_i|<\frac{\ep}{n}$  on $\W_\eta(D_{i+1} \times \R^{n-i-1})$ for all $0 \le i<n$ (starting with $g_n:=f$). The function $g_0$ will  then be the desired approximation.

Let us fix $i<n$ and assume that $g_{i+1}$ has been constructed. The cell $D_{i+1}$ can be written $(\xi_{i},\xi'_{i})$ where $\xi_{i}$ is either $-\infty$ or a $\C^\om$ Lipschitz function on $D_{i}$ and $\xi_{i}'$ is either $+\infty$ or a $\C^\om$ Lipschitz function on $D_{i}$ satisfying $\xi_i< \xi_i'$. 

 We start with the case where $\xi_i'$ and $\xi_i$ are both finite. We first check that we can assume, without loss of generality, that $D_{i+1}=D_i\times (0,1)$, i.e., that $\xi_i\equiv 0$ and $\xi'_i\equiv 1$. In fact, if $\theta:D_i \times \R^{n-i}\to D_i \times \R^{n-i}$ denotes the $\C^\om$ diffeomorphism defined by $\theta(\xt,y,t):=(\xt,\frac{y-\xi_i(\xt)}{\xi_i'(\xt)-\xi_i(\xt)},t)$, for  $(\xt,y,t) \in D_i \times \R\times \R^{n-i-1}$, it  suffices to prove the result for $ \theta_*g_{i+1}$, assuming $\xi_i\equiv 0$ and $\xi'_i\equiv 1$.

 Fix an integer $j\ge 1$, and set for $x=(\xt,y,t)\in D_{i}\times \R\times \R^{n-i-1}$:
 $$g_{i}(x)=g_{i+1}(\xt,\Psi(\dej(x),y),t),$$where $\Psi$ is provided by Lemma \ref{lem_phiep}.
 For every $x=(\xt,y,t)\in D_{i}\times \R\times \R^{n-i-1}$,   we have
 $\Psi(\dej(x),y)\in (0,1)$, which means that  $(\xt,\Psi(\dej(x),y),t)\in D_{i+1} \times \R^{n-i-1} $. Hence,  $g_{i}$ is a (well-defined) smooth function on $D_{i}\times \R^{n-i}$. 
 
    We claim that if $j$ is chosen large enough $|g_{i+1}-g_i|_1<\frac{\ep}{n}$ on $\W_\eta(D_{i+1}\times \R^{n-i-1})$. For simplicity, let for $x=(\xt,y,t)\in D_{i}\times \R\times \R^{n-i-1}$ 
 and set $$F(\xt,y,t):=(\xt,\Psi(\dej(x),y),t),$$ so that  we have $g_i=F^*g_{i+1}$ on $D_{i}\times \R^{n-i}$. We deduce that it suffices to show that we can make $|F-Id_{\R^n}|_1$  smaller than any given positive continuous function on $\W_\eta(D_{i+1}\times \R^{n-i-1})$ by choosing $j$ large enough. But, in view of (\ref{eq_psi}), this fact is clear. This proves the result in the case where $\xi_i$ and $\xi'_i$ are finite.
 

  We now address the case where one of the two functions, say $\xi_{i}'$, is infinite (if both are infinite then $D_{i+1}=D_i\times \R$ and there is nothing to prove). In this case, composing with the diffeomorphism $(\xt,y,t)\mapsto (\xt,y+\xi_i(\xt),t)$ if necessary, we see that we can assume that $\xi_i\equiv 0$.  We then define $g_{i}$ in the same way, just replacing the mapping $\Psi$ provided by Lemma \ref{lem_phiep} with the mapping $\Phi$ provided by Lemma \ref{lem_phiep1}, i.e.,  we set for $x=(\xt,y,t)\in D_{i}\times \R\times \R^{n-i-1}$:
 $$g_{i}(x):=g_{i+1}(\xt,\Phi(\dej(x),y),t).$$  The same argument (simply replacing  Lemma \ref{lem_phiep} with Lemma \ref{lem_phiep1}) then yields that if $j$ is chosen large enough then $g_i$ is a sufficiently close approximation of $g_{i+1}$ on $\W_\eta(D_{i+1}\times \R^{n-i-1})$.
\end{proof}

 \begin{thm}\label{thm_approx_lips}
 Let $f:A \to \R$ be a Lipschitz definable function, $A\subset \R^n$. For every $\ep\in \spn$ there exists a  Lipschitz function $g\in \din$ such that $|g-f|<\ep$ on $A$. Moreover, the Lipschitz constant of $g$ can be bounded independently of the chosen function $\ep$.
\end{thm}

To prove this theorem, we shall need the following two propositions which will also be used in the proof of Proposition \ref{pro_rn}.

\begin{pro}\label{pro_bump_graph}
 Given $L\in \R$, there are positive constants $N$ and $b$ such that for every $L$-Lipschitz functions  $\delta\in \D(\R^{n}, (0,1))$ and $\xi\in \D(\R^{n-1})$  and every $\mu \in \spn$ there is   $\lambda\in \D^\infty(\R^n, [0,3))$ which satisfies:
 \begin{enumerate}
 \item\label{item_v_1}    $\lambda> 1$ on  $\V_{b\delta}(\Gamma_\xi)$ and $|\lambda|_1<\mu$ on $\R^n\setminus \V_{\delta}(\Gamma_\xi)$.
   \item\label{item_v_2}  $|d\lambda|<N$ on   $\V_{b\delta}(\Gamma_\xi)$ and  $|d\lambda|< \frac{N }{ \delta}$ on $\R^n$.
 \end{enumerate}
   \end{pro}

\begin{pro}\label{pro_bump_open_cell}
 Let $C$ be a Lipschitz cell of $\R^n$ of dimension $ n$. Given $L\in \R$, there are positive constants $N$ and $c$ such that for every $L$-Lipschitz function   $\delta\in \D(\R^{n}, (0,1))$ and each $\mu\in \spn$, we can find  $\lambda\in \D^\infty(\R^n, [0,3))$ which satisfies:
 \begin{enumerate}
 \item\label{item_w_1}  $\lambda> 1$ on  $\W_{\delta}(C)$ and $|\lambda|_1<\mu$ on $\R^n\setminus \W_{c\delta}(C)$.
  \item\label{item_w_2}  $|d\lambda|<N$ on   $\W_{\delta}(C)$ and $|d\lambda|< \frac{N }{ \delta}$ on $\R^n$.
 \end{enumerate}
\end{pro}
We wish to make two remarks about the statements of these two propositions which will be useful in the proofs.
\begin{rem}\label{rem_produit_et_borne}
  We assume that $\lambda (x)\in [0,3)$ for convenience. What actually matters is that $\lambda$ is bounded independently of $\delta$ and $\mu$. This is the reason why we will not really care when the constructed function has values greater than $3$. In particular, the constructed function will be the product of such functions although it takes higher values. 
  Indeed, one can always compose the resulting function with a $\C^\infty$ function $h:\R \to \R$ that maps the image into $[0,3)$.
\end{rem} 

\begin{rem}\label{rem_ineq_estimate_der_lambda}
 In Proposition \ref{pro_bump_graph} (resp. Proposition \ref{pro_bump_open_cell}), we require $|d\lambda|< \frac{N }{ \delta}$ on $\R^n$. In the proof we will sometimes just check it on $\V_{\delta}(\Gamma_\xi) \setminus \V_{b\delta}(\Gamma_\xi)$ (resp.  $\W_{c\delta}(C)\setminus \W_{\delta}(C)$) since on the complement of this set the estimates of the derivative $|\lambda|_1<\mu$ and $|d\lambda|<N$ will be better.
\end{rem}

We are going to prove simultaneously  Theorem \ref{thm_approx_lips} and  Propositions  \ref{pro_bump_graph} and \ref{pro_bump_open_cell},  inductively on $n$. We do this because, given $n\in \N$,  on the one hand we can show that the statement of Theorem \ref{thm_approx_lips} for $(n-1)$ variable functions implies Propositions \ref{pro_bump_graph}  and \ref{pro_bump_open_cell}   in $\R^n$ (see steps  \ref{ste_sep_graph_int} and \ref{ste_sep_open_cell}  below) and we need, on the other hand,  these two statements in $\R^{n}$ to establish  Theorem \ref{thm_approx_lips} for $n$-variable functions (see step \ref{ste_ap0n}).

 
 To start our induction, we thus just have to check  Theorem \ref{thm_approx_lips}  in $\R^0=\{0\}$, which is obvious. We therefore fix $n \ge 1$ and assume that  Theorem \ref{thm_approx_lips} holds for  $(n-1)$-variable functions.

 \begin{rem}\label{rem_approx_comparaison}
  Given a Lipschitz function $\xi\in \D(\R^{n-1})$ and $\ep \in \spl(\R^{n-1})$, by our induction assumption,  we know that there is  $\zeta\in \D^\om(\R^{n-1})$ satisfying $|\zeta-\xi|<\ep$. We can actually require in addition that $\zeta<\xi$. Indeed, since we know that for every $j$ we can construct a function $\zeta_j$ satisfying $|\zeta_j-\xi| <\dej$ (see (\ref{eq_delta_j}) for $\dej$), it suffices to choose $j$ sufficiently  big to have $\dej<\frac{\ep}{2}$ and then $\zeta:=\zeta_j-\dej$ satisfies $|\zeta-\xi|<\ep$ and $\zeta<\xi$.
 \end{rem}

 \begin{step}\label{ste_cell_pos_codim_0}Given  definable  Lipschitz functions $ \xi$ and $f$  on $\R^n$, we show that there is a constant $L$ such that for each $\delta\in \D^0(\R^n,(0,1))$ there is an $L$-Lipschitz function $g\in \don$ such that $|f-g|<L\cdot \delta$ on $\V_{\delta}(\Gamma_{\xi})$.

For $x\in \R^{n-1}$, let $\theta(x):=f(x,\xi(x))$. This defines a $(1+L_\xi) \cdot L_f$-Lipschitz function. Let also 
$$\delta'(x):=\inf \{\delta(x,y) : y \in [\xi(x)-L_\xi-1,\xi(x)+L_\xi+1 ] \}.$$
Thanks to our induction hypothesis, we know that for every $\delta\in \D^0(\R^n,(0,1))$, we can find a $\C^\om$ function $g$ such that $|g-\theta|<\delta'$. Moreover, again thanks to the induction hypothesis, this function may be required to be $L$-Lipschitz, for some constant $L$ (independent of $\delta$).  
 
  We may regard the function $g$ as an $n$-variable function, constant with respect to the last variable. If $x=(\xt,x_n)\in \V_{\delta} (\Gamma_\xi)$, by  (\ref{eq_dist_au_graph}), we see that $|x_n-\xi(\xt)|\le (L_\xi+1)\delta(x)$, which, by definition of $\delta'$, entails  
  $\delta'(\xt)\le \delta(x)$ (since $\delta\le 1$).   
As a matter of fact,  for all $x=(\xt,x_n)\in \R^{n-1}\times \R$ we have
  $$|g(x)-f(x)|\le  |g(x)-f(\xt,\xi(\xt))|+|f(\xt,\xi(\xt))-f(\xt,x_n)|<\delta(\xt)+L_f|x_n-\xi(\xt)|$$
 which,  by (\ref{eq_dist_au_graph}), is smaller than $(1+L_f(L_\xi+1))\delta(x)$ if $x \in \V_{\delta} (\Gamma_\xi)$. This shows that $g$ has the required properties.
   \end{step}

  \begin{step}\label{ste_sep_graph} Given $L\in \R$, we show that there is a constant $N$ such that for every $L$-Lipschitz functions  $\alpha\in \D(\R^{n-1}, (0,1))$ and $\xi\in \D(\R^{n-1})$,  and every $\mu \in \spn$ there is $\lambda\in \D^\om(\R^n, [0,3))$ which satisfies:
 \begin{enumerate}[(i)]
  \item\label{item_mu_step_2}  $\lambda> 1$ on   $(-\infty,\xi]$ and  $|\lambda|_1<\mu$ on   $[\xi+\alpha,+\infty)$.
    \item\label{item_iii_step_2} $|d\lambda|<N$ on $(-\infty,\xi]$ and $|d\lambda|< \frac{N }{ \alpha}$ on $[\xi,+\infty)$.
 \end{enumerate}

%
%
%

Fix $L\in \R$ as well as some $L$-Lipschitz functions  $\alpha:\R^{n-1}\to (0,1)$ and $\xi:\R^{n-1}\to \R$.
By induction on $n$,  we know that we can find  $\zeta\in \St^\infty(\R^{n-1})$ such that on $\R^{n-1}$:
\begin{equation}\label{eq_zeta_xi}|\zeta-\xi|<\frac{\alpha}{2}.\end{equation}
Moreover, our induction hypothesis also ensures that  the Lipschitz constant of $\zeta$ may be assumed to be bounded independently of $\alpha$. We can also assume $\xi\le \zeta$ (see Remark \ref{rem_approx_comparaison}). 
For the same reason, there is a Lipschitz  function $\alpha'\in \D^\om(\R^{n-1})$   such that $|\alpha-\alpha'|<\frac{\alpha}{2}$, and  again, by  Remark \ref{rem_approx_comparaison}, we can assume $\alpha'\le \alpha$.

Let also $\mu'\in \spl(\R^{n})$ be a $\C^\om$-function such that $|\mu'|_1< \mu$.
 Define then a $\C^\om$ function $\lambda$  by setting for every $x=(\xt,x_n)\in \R^n$ 
 $$\lambda(x)=\psi(\mu'(x), \frac{\alpha'(\xt)}{2},x_n-\zeta(\xt)),$$where $\psi$ is provided by Lemma \ref{lem_psiep1}, and
 let us check that the function $\lambda$ has the required properties (\ref{item_mu_step_2}) and (\ref{item_iii_step_2}).

 If $x=(\xt,x_n)\in \R^n$  satisfies $x_n\le \xi(\xt)$ then $x_n \le \zeta(\xt)$, so that the first inequality of (\ref{item_mu_step_2}) follows from (\ref{item_psi_ge_1}) of Lemma \ref{lem_psiep1}.
 
 To check the second inequality,  take $x=(\xt,x_n)\in \R^n$ which satisfies   $(x_n-\xi(\xt))\ge \alpha(\xt)$ and notice that then
 $$(x_n-\zeta(\xt))\ge (x_n-\xi(\xt))-|\xi(\xt)-\zeta(\xt)|\ge \alpha(\xt)- \frac{\alpha(\xt)}{2}= \frac{\alpha(\xt)}{2}\ge \frac{\alpha'(\xt)}{2},$$
 which,
 thanks to (\ref{item_mu}) of  Lemma \ref{lem_psiep1},  yields the claimed inequality.


It thus only remains to establish (\ref{item_iii_step_2}). Notice for this purpose  that if we set $$H(x):=(\mu'(x), \frac{\alpha'(\xt)}{2},x_n-\zeta(\xt)),$$
then $H$ is a Lipschitz mapping and we have $\lambda=H^*\psi$. If $x_n\le \xi(\xt) $ then $x_n-\zeta(\xt)\le 0 $ so that, by (\ref{eq_der_psi}), $|dH(x)|\le A$ which, since $\psi$ is Lipschitz, yields the first estimate of  (\ref{item_iii_step_2}). Furthermore, notice that
   there is a positive constant $A$ such that:
\begin{equation}\label{eq_pr_N_1}|d(\psi\circ H)|\le |dH|\cdot |d\psi|\circ H\overset{\mbox{(\ref{eq_der_psi})}}{<}L_{H} \cdot \frac{2A}{\alpha'}\le L_{H} \cdot \frac{4A}{\alpha}, \end{equation}
which yields the second  estimate of  (\ref{item_iii_step_2}).
 \end{step}

 \begin{step}\label{ste_sep_graph_cell}We establish the statement of Proposition \ref{pro_bump_graph} for  $n$.

Fix a positive constant $L$ and denote by $\kappa$ the corresponding constant provided by Lemma \ref{lem_v_w}. 

By step \ref{ste_sep_graph}, there is a positive constant $N$ such that for each $L$-Lipschitz functions $\xi:\R^{n-1}\to \R$ and $\delta:\R^{n}\to (0,1)$ and each  $\mu \in \spn$, we can find a $\C^\om$ function $\lambda_1:\R^n\to [0,3)$ satisfying $\lambda_1> 1$ and $|d\lambda_1|<N$ on  $(-\infty,\xi+\alpha]$ and
    $|\lambda_1|_1<\mu$ on  $[\xi+2\alpha,+\om)$,  and satisfying   for all $x=(\xt,x_n)\in \R^{n-1}\times \R$:
    \begin{equation}\label{eq_lambda_1_der}
|d_x\lambda_1|< \frac{N }{\alpha(\xt)}, 
    \end{equation}
    where we have set for $\xt\in \R^{n-1}$, $\alpha(\xt):=\frac{\kappa}{2}\delta(\xt,\xi(\xt))$.

Step \ref{ste_sep_graph} also entails that if $N$ is large enough, we can find, for every $L$-Lipschitz functions $\xi$ and $\alpha$ on $\R^{n-1}$ and each  $\mu \in \spn$, a $\C^\om$ function $\lambda_2:\R^n\to [0,3)$ such that 
$ \lambda_2> 1$ and $|d\lambda_2|<N$ on $[\xi-\alpha,+\infty)$ and $|\lambda_2|_1<\mu$ on  $(-\infty,\xi-2\alpha] $, and for  which estimate (\ref{eq_lambda_1_der}) also holds
(since we can apply step \ref{ste_sep_graph} to the function $(-\xi+\alpha)$ to get a function $\lambda$ and set $\lambda_2(\xt,x_n):=\lambda(\xt,-x_n)$). 

But then the function $\lambda:\R^n \to [0,9]$ defined as $\lambda:=\lambda_1\cdot \lambda_2$ is a function satisfying:
\begin{enumerate}
  \item $\lambda> 1$ on $[\xi-\alpha,\xi+\alpha]$ and $|\lambda|_1<(6+N)\mu$ on  $\R^n\setminus[\xi-2\alpha,\xi+2\alpha]$.
   \item  $|d\lambda|<6N$ on $[\xi-\alpha,\xi+\alpha]$ and $|d\lambda|< \frac{6N }{ \alpha}$ on $\R^n\setminus [\xi-\alpha,\xi+\alpha]$.
 \end{enumerate}
 We  claim that this function has the required properties (see Remark \ref{rem_produit_et_borne}).  Indeed, thanks to Lemma \ref{lem_v_w}, we see that $(1)$  above implies $(i)$  of the proposition (it is enough to establish the desired statement for arbitrarily small functions $\mu$). Note also that, thanks to Lemma \ref{lem_v_w}, the first estimate in $(2)$ implies the first estimate of $(ii)$. 
 
 Moreover, if $\mu$ is chosen small enough then $(1)$ yields the second estimate that appears in $(ii)$ on the set  $\R^n\setminus[\xi-2\alpha,\xi+2\alpha]$.
 We thus only have to prove this inequality on $[\xi-2\alpha,\xi+2\alpha]$.  By Lipschitzness of $\delta$, for all $x=(\xt,x_n)$ in this set we have
 $|\delta(x)-\delta(\xt,\xi(\xt))|\le 2L\alpha(\xt)$, which implies that  for such $x$ 
 $$\delta(x) \le \delta(\xt,\xi(\xt))+ 2L\alpha(\xt)=(\frac{2}{\kappa}+2L)\cdot\alpha(\xt),$$
 which by the second inequality of $(2)$ just above, establishes the second estimate of $(ii)$. 
 \end{step}

    \begin{step}\label{ste_sep_graph_int}
     We establish the statement of  Proposition \ref{pro_bump_open_cell} in the case where the cell $C$ is of type $(\xi_0,\xi_1)$, where either $\xi_1$ is either $+\infty$ or a $\C^\om$ Lipschitz function on $\R^{n-1}$ and $\xi_0$ is  $-\infty$ or a $\C^\om$ Lipschitz function on $\R^{n-1}$   satisfying $\xi_0< \xi_1$.


We first assume that both $\xi_1$ and $\xi_0$ are finite. We shall use the same strategy as in the preceding step. The only difference is that we are now working with two functions $\xi_1$ and $\xi_0$ instead of one single function $\xi$. Fix a positive constant $L\ge \max(L_{\xi_0},L_{\xi_1})$ as well as an $L$-Lipschitz function  $\delta:\R^{n}\to (0,1)$ and   $\mu \in \spn$. Define then two $(n-1)$-variable functions by setting for $\xt\in \R^{n-1}$, $$\alpha_1(\xt):=\frac{\kappa}{2}\delta(\xt,\xi_1(\xt))\et \alpha_0(\xt):=\frac{\kappa}{2}\delta(\xt,\xi_0(\xt)),$$ where $\kappa$ is provided by Lemma \ref{lem_v_w}.
 
 By step \ref{ste_sep_graph},  we know that we can find a $\C^\om$ function $\lambda_1:\R^n\to [0,3)$ satisfying $\lambda_1> 1$ on the set $(-\infty,\xi_1-2\alpha_1]$ and
    $|\lambda_1|_1<\mu$ on the set $[\xi_1-\alpha_1,+\infty)$, and such that
    \begin{equation}\label{eq_dlambda_1}
      \mbox{$|d\lambda_1|<N$ on $(-\infty,\xi_1-2\alpha_1]\quad$  and $\quad|d\lambda_1|< \frac{N }{ \alpha_1}$ on $[\xi_1-2\alpha_1,+\infty)$}.
    \end{equation}


Step \ref{ste_sep_graph} also entails that if $N$ is chosen large enough, we can find a $\C^\om$ function $\lambda_0:\R^n\to [0,3)$ such that 
$ \lambda_0> 1$ on $[\xi_0+2\alpha_0,+\infty)$ and $|\lambda_0|_1<\mu$ on  $\R^n\setminus [\xi_0+\alpha_0,+\infty)$ (since we can apply step \ref{ste_sep_graph} to the function $(-\xi_0-2\alpha_0)$ to get a function $\lambda$ and set $\lambda_0(\xt,x_n):=\lambda(\xt,-x_n)$), and such that \begin{equation}\label{eq_dlambda_2}
      \mbox{$|d\lambda_0|<N$ on $[\xi_0+2\alpha_0,+\infty)\quad$  and $\quad|d\lambda_0|< \frac{N }{ \alpha_0}$ on $(-\infty,\xi_0+2\alpha_0]$}.
    \end{equation}

We claim that the function $\lambda:=\lambda_1\cdot \lambda_0$ has the desired properties (see Remark \ref{rem_produit_et_borne}). Indeed, since $|\lambda|_1<(6+\frac{N}{\delta})\mu$ on the set $(-\infty, \xi_0+\alpha_0 ]$ and on $[ \xi_1-\alpha_1 ,+\infty)$, it follows from Lemma \ref{lem_v_w} that $|\lambda|_1<(6+\frac{N}{\delta})\mu$ on $\R^n \setminus \W_{c\delta}(C)$, for $c$ small enough (it is enough show that this inequality can be obtained for arbitrarily small functions $\mu$). As in addition we have  $\lambda>1$ on  $[\xi_0+2\alpha_0, \xi_1-2\alpha_1 ]$ which, by Lemma \ref{lem_v_w} (and choices of $\alpha_0$ and  $\alpha_1$), contains $\W_{\delta}(C)$, we can see that $(\ref{item_w_1})$ holds.


By (\ref{eq_dlambda_1}) and (\ref{eq_dlambda_2}), we see that $|d\lambda|<6N$ on  $[\xi_0+2\alpha_0,\xi_1-2\alpha_1]$, which, thanks to Lemma \ref{lem_v_w}, already yields the first part of $(\ref{item_w_2})$.  

It suffices to prove the second inequality of $(\ref{item_w_2})$  on $(\xi_0,\xi_1)$ (see Remark \ref{rem_ineq_estimate_der_lambda}). Since it is enough to establish it for both $\lambda_1$ and $\lambda_0$, we will focus on  $\lambda_1$, the corresponding argument for  $\lambda_0$ being completely analogous.

By Lipschitzness of $\delta$, for all $x=(\xt,x_n)$ in $(\xi_1-2\alpha_1,\xi_1)$ we have 
 $|\delta(x)-\delta(\xt,\xi_1(\xt))|\le 2L\alpha_1(\xt)$, which implies that  for such $x$ 
 $$\delta(x) \le \delta(\xt,\xi_1(\xt))+ 2L\alpha_1(\xt)=(\frac{2}{\kappa}+2L)\alpha_1(\xt),$$
 which by (\ref{eq_dlambda_1})  entails that $|d\lambda_1|<\frac{N(\frac{2}{\kappa}+2L)}{\delta}$ on $(\xi_1-2\alpha_1,\xi_1)$. On $(-\infty,\xi_1-2\alpha_1)$, this follows from the first inequality of (\ref{eq_dlambda_1}).
 
 
 To complete the proof of step \ref{ste_sep_graph_int}, note that in the case where  $\xi_1$ (resp. $\xi_0$) is infinite, the function $\lambda_1$ (resp. $\lambda_0$) has all the required properties.
\end{step}
 
  \begin{step}\label{ste_sep_open_cell}We establish the statement of Proposition \ref{pro_bump_open_cell} for  $n$. 

Let $D_i:=\pi_i(C)$, $0\le i\le n$, where $\pi_i :\R^n\to \R^i$ is the canonical projection. For every $i\ge 1$, we can write $D_i$ as $(\xi_i,\xi'_i)$  where $\xi_{i}$ is either $-\infty$ or a $\C^\om$ Lipschitz function on $D_{i-1}$ and $\xi_{i}'$ is either $+\infty$ or a $\C^\om$ Lipschitz function on $D_{i-1}$ satisfying $\xi_i< \xi_i'$. 

We can extend $\xi_i$ and $\xi'_i$ to Lipschitz functions $\overline{\xi}_i$ and $\overline{\xi}'_i$ on $\R^{i-1}$ satisfying $\overline{\xi}_i\le \overline{\xi}'_i$, and then, regarding these extensions as constant with respect  to the $(n-i)$ last variables, to Lipschitz functions on $\R^{n-1}$. Set then $$C_i:=\{(x_1,\dots,x_n)\in \R^n:\overline{\xi}_i(x_1,\dots , \hat{x_i} ,\dots,x_{n-1} ) <x_i<\overline{\xi}'_i(x_1,\dots , \hat{x_i} ,\dots,x_{n-1} ) \},$$
where $\hat{x_i}$ means that the coordinate $x_i$ is omitted. 

It now follows from Step \ref{ste_sep_graph_int} that for each $i$, each $L$-Lipschitz function  $\delta$ and each $\mu\in \spn$, there is a $\C^\om$ function $\lambda_i$ on $\R^n$ satisfying 
 $\lambda_i> 1$ on the set $ \W_{\delta}(C_i)$ and 
  $|\lambda_i|_1<\mu$ on the set $\R^n\setminus \W_{c\delta}(C_i)$, and satisfying $|d\lambda_i|\le N$ on  $ \W_{\delta}(C_i)$ as well as  $|d\lambda_i|< \frac{N }{ \delta}$  (for some constants $c$ and $N$ independent of $\mu$ and $\delta$). Indeed,  Step \ref{ste_sep_graph_int} ensures that this fact holds for $C_n$, and, since we can interchange the $i^{th}$ and the $n^{th}$ coordinates, we see that this is true as well for all the $C_i$. 
  
But since $C=\cap_{i=1} ^n C_i$, this implies that the function $\lambda:=\lambda_1\cdots \lambda_n$ has all the required properties (see Remark \ref{rem_produit_et_borne}). 
  \end{step}

  \begin{step}\label{ste_ap0n}We prove the statement of Theorem \ref{thm_approx_lips} for $n$, completing the induction step.
  
   Fix a Lipschitz function $f:A\to \R$,   $A\subset \R^n$. The function $f$ can be extended to an $L_f$-Lipschitz function on $\R^n$ (still denoted $f$). For $j\in \N$ large enough, we shall construct a Lipschitz  $\C^\om$ function $g_j:\R^n\to \R$ satisfying $|g_j-f|<a\cdot \dej$, with $a\in \R$ independent of $j$.
As the Lipschitz constant of $g_j$ will be bounded independently of $j$, this will be enough for our purpose.   
   
  Let $\Sigma$ be a stratification  of $\R^n$  as provided by Proposition \ref{pro_stratification}. We denote by $\Sigma'$ the collection of all the strata  of $\Sigma$ that are maximal, in the sense that they do not lie in the closure of another stratum. We denote by
   $\Sigma_0$ the set constituted by all the strata of $\Sigma'$ of dimension $n$ and by 
$\Sigma_1$ the set of the strata of $\Sigma'$ that are of positive codimension. 
Hence, $\Sigma'=\Sigma_0\cup \Sigma_1$.

 Let $S\in \Sigma_1$. Up to an orthonormal change of coordinates, $S$ is the graph of some Lipschitz function $\xi_S:D\to \R$, where $D$ is a Lipschitz cell of $\R^{n-1}$. As no confusion may arise, we will identify  $S$ with the graph of $\xi_S$. We can extend $\xi_S$ to an $L_{\xi_S}$-Lipschitz function $\overline{\xi}_S$ defined on $\R^{n-1}$.
 
 By step \ref{ste_cell_pos_codim_0}, there is a constant $L_S$ such that  for every $S\in \Sigma_1$  and every $j$, there is a $\C^\om$ function $g_{S,j}$ which satisfies on $\V_{\dej} (\Gamma_{\xib_S})$ \begin{equation}\label{eq_g_f_sigma_2}
      |g_{S,j}-f|<L_S\cdot  \dej  .                                                                                                                                                                                                                                                                                                                                                                                                                   \end{equation}
      
%

Let $L$ be a constant bigger than all the  $L_{\xi_T}$, $T\in \Sigma_1$, and let $b$ be the constant obtained by applying Proposition \ref{pro_bump_graph} to this $L$.  
 Moreover, applying Proposition \ref{pro_bump_open_cell} to all the strata of $\Sigma_0$ (which are open Lipschitz cells after a possible orthonormal change of coordinates), we get a positive constant $c$.

Fix now a stratum  $S\in \Sigma_0$. By Lemma \ref{lem_app_open_cell}, for every $j\in \N$, there is a $\C^\om$ function $g_{S,j}$ on $\R^n$ which satisfies on $\W_{bc\dej}(S) $ \begin{equation}\label{eq_g_f_sigma_1}                                                                                                                                                                                                                                                                                                                                                                                                                                                                         
                                               |g_{S,j}-f|_1<\dej.                                                                                                                                                                                                                                                                                                                                                                                                                            \end{equation}
Now, for  $j$ sufficiently big we can set:  
  \begin{equation}\label{eq_dfn_mu}
   \mu_j:=\min_{S\in \Sigma'} \frac{\dej}{1+|f|+|g_{S,j}|_1}.
  \end{equation}

 We then are going to use the bump functions provided by Propositions \ref{pro_bump_graph} and \ref{pro_bump_open_cell} to construct the desired approximation by means of the $g_{S,j}$. 
 
Let for this purpose $S$ be an element of $\Sigma_1$. As above, we will identify $S$  with the graph of $\xi_S$.
By  Proposition \ref{pro_bump_graph},  there exists a constant $N_S\ge 1$ such that for every $j$ 
there is a $\C^\om$ function $\lambda_{S,j}:\R^n\to [0,3)$ such that 
 $|\lambda_{S,j}|_1<\mu_j$ on the set $\R^n\setminus\V_{\dej}(\Gamma_{\xib_S})$,
 $\lambda_{S,j}> 1$ on the set $\V_{b\dej}(\Gamma_{\xib_S})$, and on $\R^n$
 \begin{equation}\label{eq_dlambda}|d\lambda_{S,j}|< \frac{N_S }{ \delta_j}.\end{equation}

Let now $S$ denote a stratum of $\Sigma_0$. By  Proposition \ref{pro_bump_open_cell}, there is a constant $N_S$ such that for every $j\in \N$  there is a $\C^\om$ function  which satisfies $|\lambda_{S,j}|_1<\mu_j$ on the set $\R^n\setminus \W_{bc\delta_j}(S)$, $\lambda_{S,j}> 1$ on the set $ \W_{b\delta_j}(S)$, and which satisfies (\ref{eq_dlambda}) for some constant $N_S$.
  
 

Set now \begin{equation}\label{eq_dfn_g}
g_j:= \sum_{S\in \Sigma'} \theta_{S,j}\cdot g_{S,j}\quad \mbox{ where }\;\;\theta_{S,j} :=\frac{\lambda_{S,j}}{\sum_{T\in \Sigma'}\lambda_{T,j}}.         
        \end{equation}

 Let us check that $g_j$ is the desired approximation. Observe first that the sets $\V_{b\dej}(\Gamma_{\xib_T})$, $T\in \Sigma_1$, together with the sets $\W_{b\dej}(T)$, $T\in \Sigma_0$, cover $\R^n$. Therefore, for every $x\in \R^n$, there is a stratum $T\in \Sigma'$ such that $\lambda_{T,j}(x)>1$. This shows that $\sum_{T\in \Sigma'}\lambda_{T,j}>1$, which proves that $\theta_{S,j}\le \lambda_{S,j}$, for each $S\in \Sigma'$.   By (\ref{eq_dfn_mu}), this implies that for every $S\in \Sigma_1$ we have  on $\R^n \setminus \V_{\dej}(\Gamma_{\xib_S})$
 $$\theta_{S,j}|g_{S,j}-f|< \dej.$$
  By (\ref{eq_g_f_sigma_2}), we deduce that  on $\R^n$ we have for such a stratum:
 \begin{equation}\label{eq_g_f_lipschitz} \theta_{S,j}|g_{S,j}-f|<L_S'\cdot \dej,\end{equation}
with $L'_S:=\max(L_S,1)$. Moreover, the same argument (replacing (\ref{eq_g_f_sigma_2})  with (\ref{eq_g_f_sigma_1})) yields the same estimate for the strata of $\Sigma_0$.
As a matter of fact, we can write:  
  $$|g_j-f|=|\sum_{S\in \Sigma'}\theta_{S,j}\cdot(g_{S,j}-f)|\overset{(\ref{eq_g_f_lipschitz})}{<}a \cdot \dej,$$
 for some constant $a$ (independent of $j$).  It thus remains to establish the Lipschitz character of $g_j$. We shall provide a bound for its derivative.

 Observe for this purpose that by definition of $g_j$ we have for $x$ in an open dense subset of $\R^n$:
 \begin{equation}\label{eq_g_f_lips}
   d(g_j-f) =d\sum_{S\in \Sigma'} \theta_{S,j} (g_{S,j}- f)
  = \sum_{S\in \Sigma'} \theta_{S,j}\cdot d(g_{S,j}- f) +(g_{S,j}- f)\cdot d \theta_{S,j}.
\end{equation} 
We first check that $ \theta_{S,j}|d(g_{S,j}- f)|$ is bounded independently of $j$. This is clear if $S$ is a stratum of $\Sigma_1$ since both $\theta_{S,j}$  and $ |d(g_{S,j}- f)|$ are bounded. In the case where $S\in \Sigma_0$, by (\ref{eq_g_f_sigma_1}), we see that $ \theta_{S,j}|d(g_{S,j}- f)|$ is bounded on  $\W_{bc\delta_j}(S)$. On the complement of this set, because $|\theta_{S,j}|<\mu_j$, by (\ref{eq_dfn_mu}), we see that the result is also clear.
It thus suffices to check that $|(g_{S,j}- f)\cdot d \theta_{S,j}|$ is bounded as well.

We first deal with the case where  $S$ belongs to $\Sigma_1$. Remark that, thanks to (\ref{eq_dlambda}), a straightforward computation of derivative shows that 
   there is a constant $N_S'$ (independent of $j$) such that we have:
\begin{equation}\label{eq_theta_S_Sigma_1_I}
  |d\theta_{S,j}|\le \frac{ N_S'}{\dej}.
\end{equation}
 By (\ref{eq_g_f_sigma_2}), we see that 
this inequality already shows that  $(g_{S,j}- f)\cdot |d \theta_{S,j}|$ is bounded on $\V_{\dej}(\Gamma_{\xib_S})$.

On the complement of this set, because we have $|\lambda_{S,j}|_1<\mu_j$, a direct computation of derivative shows  that there is a constant $N''_S$ such that we have 
\begin{equation}\label{eq_theta_S_Sigma_1_II}
   |d\theta_{S,j}|<\mu_j \cdot \frac{N_S''}{\dej}\overset{(\ref{eq_dfn_mu})}{\le}  \frac{N_S''}{1+|g_{S,j}|+|f|},
\end{equation}
which clearly entails that $|(g_{S,j}- f)\cdot d \theta_{S,j} |$ is bounded on $\R^n\setminus \V_{\dej}(\Gamma_{\xib_S})$.

 We now address the case of a stratum $S$ of $\Sigma_0$. On $\W_{bc\delta_j}(S)$, we have $|g_{S,j}-f|<\delta_j$, an since  (\ref{eq_theta_S_Sigma_1_I}) holds for the strata of $\Sigma_0$ as well, we see that $(g_{S,j}- f)\cdot d \theta_{S,j}$ is bounded on this set.
Moreover, on the complement of this set, because we have $|\lambda_{S,j}|_1<\mu_j$ we see that (\ref{eq_theta_S_Sigma_1_II}) holds for $S$, which entails  that $ |(g_{S,j}- f)\cdot d \theta_{S,j} | $ is a bounded function.
%
%
  \end{step}

  \end{section}

\begin{section}{$\C^\om$-approximations of $\C^1$ functions}
We prove in this section that $\D^\om(M)$ is dense in $\D^1(M)$, for $M$ $\C^\om$ submanifold of $\R^n$. We will first prove it  in the case where the considered manifold is $\R^n$ (Proposition \ref{pro_rn}) in order to derive it for an arbitrary  definable submanifold of $\R^n$ (Theorem \ref{thm_ef_C_1}).

\subsection{The case $M=\R^n$.} The strategy is somehow similar to the one used in the proof of Theorem \ref{thm_approx_lips}. We shall however need  approximations with sharp estimates on the derivative.  This motivates the following definition.
\begin{dfn}
Let  $f\in \D^1(\R^n)$ and take $\ep$ and $\delta$ in $\spn$.
 Given $A\subset \R^n$, we say that a $\C^\om$ function $g:\R^n\to \R$ induces a {\it $(\delta,\ep)$-approximation of $f$  on $A$} if we have $|f-g|<\delta \cdot \ep$ and  $| df-dg|< \ep$  on  $A$.
 \end{dfn}

\begin{pro}\label{pro_rn}
 $\din$ is dense in $\D^1(\R^n)$ for the $\C^1$ Efroymson topology.
\end{pro}

\begin{proof} We prove the result by induction on $n$, starting at $n=0$, for which the statement is vacuous. Let $f\in \D^1(\R^n)$ and $\ep\in \spn$. We split the induction step into two steps.
\begin{ste}\label{ste_cell_r_n-1} Given a  Lipschitz function $\xi\in \D(\R^{n-1})$ and $ \ep \in \spn$, we show that  there exists $\eta \in \spn$ such that for each   $\delta\in \D^1(\R^n,(0,1))$ satisfying $|\delta|_1\le \eta$,  there is $g\in \din$ which induces a $(\delta,\ep)$-approximation of $f$ on $[\xi-\delta',\xi+\delta']$, where $\delta'(\xt):=\delta(\xt,\xi(\xt))$.

 We first reduce it  to the case $\xi\equiv 0$, i.e., we establish the claimed statement in general assuming it to be true in the particular case where $\xi$ is identically zero.
By Theorem \ref{thm_approx_lips},  there is a $\C^\om$ Lipschitz function $\xi'$ on $\R^n$ satisfying $|\xi-\xi'|<\delta'$ on $\R^{n-1}$, where \begin{equation}\label{eq_ep_prime}
                                                               \ep'(\xt):=\inf\{\ep(\xt,x_n): x_n\in [\xi(\xt)-1,\xi(\xt)+1]\}.
                                                               \end{equation}
         Moreover, we can assume $\xi'$ to be Lipschitz with a Lipschitz constant bounded independently of $\delta$ and $\ep$.

Let us define a $\C^\om$ diffeomorphism $\theta:\R^n\to \R^n$ by $\theta(\xt,x_n)=(\xt,x_n-\xi'(\xt))$ and    set  $u:= \theta_{*}f$.  Apply the case $\xi\equiv 0$ to the function $u$ to get a function $v\in \din$ which induces a $(2\theta_{*}\delta,\frac{\theta_{*}\ep}{2})$-approximation of $u$  on $[-2\delta',2\delta']$ (for each $\delta$ with $|\delta|_1$ sufficiently small).
Let then $g:=\theta^{*}v$ and observe that  since $|v-u|<\theta_{*}\ep\cdot \theta_{*}\delta$ on $[-2\delta',2\delta']$, we already see that $ |g-f|=|\theta^*(v-u)|<\ep\cdot \delta$ on 
$[\xi'-2\delta',\xi'+2\delta']$. Since $|\xi-\xi'|<\delta'$, we also see that $[\xi'-2\delta',\xi'+2\delta']$ contains $[\xi-\delta',\xi+\delta']$.
Moreover, 
$$
 |d (g-f)|=|d (\theta^*(v-u))|\le  |d \theta|\cdot |d (v- u)|\circ \theta <  L_\theta\cdot \ep .
$$ As $L_\theta$ is bounded independently of $\delta$ and $\ep$, this completes our reduction.

We now shall prove the result in the case where $\xi\equiv 0$. Fix $\ep \in \spn$ and let $\delta \in \spn$ be a $\C^1$ function satisfying $|\delta|_1<\frac{1}{2}$.  Such a function being $\frac{1}{2}$-Lipschitz, it is easily checked that on $[-\delta',\delta']$ we have:
\begin{equation}\label{eq_delta_deltaprime}
 \delta'(\xt)\le2\delta(x).
\end{equation}

 By Lemma \ref{lem delta ep}, if $\nu \in \spn$ is a small enough function then for all $x$ and $x'$ in $\R^n$ satisfying $|x-x'|<\nu(x)$ we have: 
\begin{equation}\label{eq_lambda_vert}
 |f(x)-f(x')|<\ep(x) \et  |d_xf-d_{x'}f|\le \ep(x).
\end{equation}
 For $\xt \in  \R^{n-1}$, let
 $$\lambda_{0}(\xt):=f(\xt,0)\et \lambda_{1}(\xt):=\frac{f(\xt,\delta'(\xt))-f(\xt,0)}{\delta'(\xt)}.$$
 By  induction on $n$,  there are  $\mu_{0}$ and $\mu_{1}$ in $\D^\om(\R^{n-1})$ such that  we have:
\begin{equation}\label{eq_mu_lambda}
|\mu_{0}- \lambda_{0}|_1<\ep'\cdot \delta'\et |\mu_{1}- \lambda_{1}|_1<\ep'\cdot \delta',
\end{equation}
where $\ep'$ is as in (\ref{eq_ep_prime}) (with $\xi\equiv 0$). Define then   $g\in \din$ by setting
$$g(\xt,x_n):= \mu_{0}(\xt)+x_n\cdot \mu_{1}(\xt).$$
 We shall show that if $|\delta|_1$ is sufficiently small then the function $g$ has the desired properties, i.e., we shall check that for $\delta$ small enough  we have on $(-\delta',\delta')$:
 \begin{equation}\label{eq_f_g}
|f-g|< \delta\cdot \ep \et |d f-d g|< \ep .  
 \end{equation}

Observe first that it follows from (\ref{eq_lambda_vert}) that for $\delta$ small enough we have for all $x=(\xt,x_n)\in (-\delta',\delta')$:
\begin{equation}\label{eq_lambda_moins_der}
|\lambda_{1}(\xt)-\frac{\pa f}{\pa x_n}(\xt,x_n)| < \ep(x). 
\end{equation}

Applying Taylor's formula to $f$, we see that   there is a function $c$ on $\R^{n-1}$  with  $|c(\xt)|\le \delta'(\xt)$ such that  for all $x=(\xt,x_n)\in (-\delta',\delta')$:
\begin{equation}\label{eq_taylor_f_alpha}
f(\xt,x_n)=\lambda_{0}(\xt)+ x_n\cdot \frac{\pa f}{\pa x_n}(\xt,c(\xt)) .
\end{equation}
 For $|x_n|<\delta'(\xt)$, thanks to the definition of $g$, we can deduce 
$$|f(x)-g(x)|\le |\lambda_{0}(\xt)-\mu_{0}(\xt)| +x_n| \frac{\pa f}{\pa x_n}(\xt,c(\xt))-\mu_{1}(\xt) |,$$
By (\ref{eq_mu_lambda}), the first term of the sum which appears in the right-hand-side of this inequality is smaller than $\ep'(\xt)\cdot \delta'(\xt)$, and by (\ref{eq_lambda_moins_der}), the second one is smaller than $\ep(x)\cdot \delta'(\xt)$, for all $x=(\xt,x_n)\in (-\delta',\delta')$. Since $\delta'(\xt)<2\delta(x)$ (by (\ref{eq_delta_deltaprime})) and $\ep'(\xt)\le \ep(x)$ (by (\ref{eq_ep_prime})) on $(-\delta',\delta')$, this shows that $|f(x)-g(x)| <2\ep(x)\cdot \delta(x)$ on this set.
   
   It thus only remains to prove the corresponding estimate for the derivative, i.e., that $|d_x f -d_x g|<\ep(x) $ on $(-\delta',\delta')$ for $|\delta|_1$ small enough.
We first check that, for $|\delta|_1$  small enough, we have for each $i< n$:\begin{equation}\label{eq_der_mu_1}
                                            | \frac{\pa \mu_{1}}{\pa x_i}|<\frac{\ep}{\delta}.
                                            \end{equation}
 To see this,  observe that a computation of derivative yields that for each $i<n$:  
 $$\delta'(\xt)\frac{\pa \lambda_{1}}{\pa x_i}(\xt)=u(\xt)-\frac{\pa \delta'}{\pa x_i}(\xt)\left(\lambda_{1}(\xt)-\frac{\pa f}{\pa x_n}(\xt,\delta'(\xt))\right),$$
 where,  we have set for simplicity $u(\xt):=\frac{\pa f}{\pa x_i}(\xt,\delta'(\xt))-\frac{\pa f}{\pa x_i}(\xt,0)$. Since $f$ is $\C^1$, the function  $u$  tends to zero (in $\D^0(\R^{n-1})$) as $|\delta|_1$ tends to $0$ and the sequence function  $\left(\lambda_{1}(\xt)-\frac{\pa f}{\pa x_i}(\xt,\delta'(\xt))\right)$ is bounded (in $\D^0(\R^{n-1})$). Consequently, $\delta'(\xt)\frac{\pa \lambda_{1}}{\pa x_i}$ tends to zero in $\D^0(\R^{n-1})$ as $|\delta|_1$ goes to zero. For $|\delta|_1$ small enough, this function is thus smaller than $\ep$, which, via the second inequality of (\ref{eq_mu_lambda}), establishes  (\ref{eq_der_mu_1}).

Observe then that by definition of $g$  we have for each $i<n$:
\begin{equation}\label{eq_taylor_g_alpha}
\frac{\pa g}{\pa x_i}(\xt,x_n)=\frac{\pa \mu_{0}}{\pa x_i}(\xt)+x_n\cdot \frac{\pa \mu_{1}}{\pa x_i}(\xt),
\end{equation}
and therefore $$|\frac{\pa g}{\pa x_i}(x)-\frac{\pa f}{\pa x_i}(x)|\le |\frac{\pa \mu_{0}}{\pa x_i}(\xt)-\frac{\pa \lambda_0}{\pa x_i}(\xt)|+|x_n\cdot \frac{\pa \mu_{1}}{\pa x_i}(\xt)|+|\frac{\pa \lambda_0}{\pa x_i}- \frac{\pa f}{\pa x_i}(x)|,$$
which, due to (\ref{eq_mu_lambda}), (\ref{eq_der_mu_1}), and (\ref{eq_lambda_vert}), must be smaller than $3\ep$ for all $(\xt,x_n)$ satisfying  $|x_n|<\delta(x)$ and each $i<n$. Finally, note that $$|\frac{\pa g}{\pa x_n}(x)-\frac{\pa f}{\pa x_n}(x)|=|\mu_{1}(\xt)-\frac{\pa f}{\pa x_n}(x)|< |\mu_{1}(\xt)-\lambda_{1}(\xt)|+|\lambda_{1}(\xt)-\frac{\pa f}{\pa x_n}(x)|, $$
which, thanks to
(\ref{eq_mu_lambda}) and (\ref{eq_lambda_moins_der}), must be smaller than $2\ep$ for all $x\in (-\delta',\delta')$. This yields the result in the case where $\xi\equiv 0$, completing step \ref{ste_cell_r_n-1}.

 \end{ste}
  
  \begin{ste}\label{ste_ap1n} We perform the induction step of the proof.


We shall make use of the same method as in the last step of the proof of Theorem \ref{thm_approx_lips}. 
   Let $\Sigma$ be a stratification  of $\R^n$   as provided by Proposition \ref{pro_stratification}. We denote by $\Sigma'$ the collection of all the strata  of $\Sigma$ that are maximal, in the sense that they do not lie in the closure of another stratum of positive codimension. We denote by
   $\Sigma_0$ the set constituted by all the strata of $\Sigma'$ of codimension $0$ and by 
$\Sigma_1$ the set of the strata of $\Sigma'$ that are of positive codimension. 
Hence, $\Sigma'=\Sigma_1\cup \Sigma_0$.

 Let $S\in \Sigma_1$. Up to an orthonormal change of coordinates, $S$ is the graph of some Lipschitz function $\xi_S:D\to \R$, where $D$ is a Lipschitz cell of $\R^{n-1}$. As no confusion may arise, we will identify  $S$ with the graph of $\xi_S$. We can extend $\xi_S$ to a Lipschitz function $\overline{\xi}_S$ defined on $\R^{n-1}$. Given $j$, set $\delta'_{S,j}(\xt):=\dej(\xt,\xib_S(\xt))$ (where  $\dej$ is as in (\ref{eq_delta_j})).

 By step \ref{ste_cell_r_n-1},  for every $j$ large enough there is a $\C^\om$ function $g_{S,j}$ which satisfies on $(\xib_S-\delta'_{S,j},\xib+\delta'_{S,j})$:
 \begin{equation}\label{eq_delta_ep_approx}
  |g_{S,j}-f|<\ep \cdot \dej \et |dg_{S,j}-df|<\ep.
 \end{equation}
By Lemma \ref{lem_v_w}, there is a positive real number $r<1$ such that $\V_{r\dej}(\Gamma_{\xib_S}) \subset (\xib_S-\delta'_{S,j},\xib+\delta'_{S,j}) $ (since the stratification $\Sigma_1$ is finite, we may choose the same $r$ for all the strata).
      
%

Fix now a stratum  $S\in \Sigma_0$. Let $L$ be a constant bigger than all the  $L_{\xi_T}$, $T\in \Sigma_1$, and let $b$ be the constant obtained by applying Proposition \ref{pro_bump_graph} to this $L$.  
 Moreover, applying Proposition \ref{pro_bump_open_cell} to each stratum of $\Sigma_0$ (which are open Lipschitz cells after a possible orthonormal change of coordinates), we get a positive constant $c$.

By Lemma \ref{lem_app_open_cell}, for every $j$, there is a $\C^\om$ function $g_{S,j}$ on $\R^n$ which satisfies  on $\W_{sc\dej}(S) $,  $s:=br$, 
\begin{equation}\label{eq_g_f_sigma_1_1}                                                                                                                                                                                                                                                                                                                                                                                                                                                                      
                                               |g_{S,j}-f|_1<\ep \cdot \dej.                                                                                                                                                                                                                                                                                                                                                                                                                            \end{equation}
Now, for  $j$ sufficiently big we can set:  
  \begin{equation}\label{eq_dfn_mu_1}
   \mu_j:=\min_{S\in \Sigma'} \frac{\dej\cdot \ep}{1+|f|_1+|g_{S,j}|_1}.
  \end{equation}

Let now again $S$ be an element of $\Sigma_1$. As above, we will identify $S$  with the graph of $\xi_S$.
By  Proposition \ref{pro_bump_graph},  there exists a constant $N_S\ge 1$ such that for every $j$ 
there is a $\C^\om$ function $\lambda_{S,j}:\R^n\to [0,3)$ such that 
 $|\lambda_{S,j}|_1<\mu_j$ on the set $\R^n\setminus\V_{r\dej}(\Gamma_{\xib_S})$,
 $\lambda_{S,j}> 1$ on the set $\V_{s\dej}(\Gamma_{\xib_S})$,  and which satisfies (\ref{eq_dlambda}) on $\R^n$.

Let now $S$ denote a stratum of $\Sigma_0$. By  Proposition \ref{pro_bump_open_cell}, there is a constant $N_S$ such that for every $j\in \N$  there is $\lambda_{S,j}\in \din$  which satisfies $|\lambda_{S,j}|_1<\mu_j$ on the set $\R^n\setminus \W_{sc\delta_j}(S)$, $\lambda_{S,j}> 1$ on the set $ \W_{s\delta_j}(S)$, and which satisfies (\ref{eq_dlambda}) for some constant $N_S$.
  
 

%

 We now define the functions $g_j$, $g_{S,j}$, and $\theta_{S,j}$ as in (\ref{eq_dfn_g}) (for some $j$ large enough). Let us check that $g_j$ is the desired approximation.

 Observe first that the sets $\V_{s\dej}(\Gamma_{\xib_T})$, $T\in \Sigma_1$, together with the sets $\W_{s\dej}(T)$, $T\in \Sigma_0$, cover $\R^n$. Therefore, for every $x\in \R^n$, there is a stratum $T\in \Sigma'$ such that $\lambda_{T,j}(x)>1$. This shows that $\sum_{T\in \Sigma'}\lambda_{T,j}(x)>1$, which proves that $\theta_{S,j}\le \lambda_{S,j}$, for each $S\in \Sigma'$.   We deduce that for every $S\in \Sigma_1$ we have  on $\R^n \setminus \V_{r\dej}(\Gamma_{\xib_S})$ 
 $$\theta_{S,j}|g_{S,j}-f|\overset{(\ref{eq_dfn_mu_1})}{<} \dej\cdot \ep<\ep.$$
  By  (\ref{eq_delta_ep_approx}), we deduce that  we have for such $S$ (on the whole of $\R^n$):
 \begin{equation}\label{eq_g_f_lipschitz_1} \theta_{S,j}|g_{S,j}-f|< \ep.\end{equation}
 Moreover,  the same argument (replacing (\ref{eq_delta_ep_approx}) with (\ref{eq_g_f_sigma_1_1}), yields the analogous estimate holds for the strata of $\Sigma_0$.
As a matter of fact, we can write:  
  $$|g_j-f|=|\sum_{S\in \Sigma'}\theta_{S,j}\cdot(g_{S,j}-f)|\overset{(\ref{eq_g_f_lipschitz_1})}{<}a \cdot \ep,$$
 for some constant $a$ (independent of $\ep$). This already yields the desired estimate for $|g_{j}-f|$. It remains to prove the analogous estimate for $|dg_j-df|$.\end{ste}

In view of (\ref{eq_g_f_lips}), it is clear that it suffices to show that for all $S\in \Sigma'$ we have for 
\begin{equation}\label{eq_claim_g_f_dtheta}
  |d\theta_{S,j}|\cdot  |g_{S,j}- f| <N\cdot\ep
\end{equation}
and \begin{equation}\label{eq_claim_g_f_theta}
  \theta_{S,j}\cdot  |dg_{S,j}- df| <N\cdot \ep,
\end{equation}
for some constant $N$ independent of $\ep$.

  We first focus on (\ref{eq_claim_g_f_dtheta}), starting with the case where the stratum $S$ belongs to $\Sigma_1$. Remark that, because (\ref{eq_dlambda}) holds for $S$, inequality (\ref{eq_theta_S_Sigma_1_I}), which comes down from this estimate, must hold as well. By (\ref{eq_delta_ep_approx}), we see that 
this inequality already shows that estimate  (\ref{eq_claim_g_f_dtheta}) holds on $\V_{r\dej}(\Gamma_{\xib_S})$.


On the complement of this set, because we have $|\lambda_{S,j}|_1<\mu_j$, a direct computation of derivative shows  that there is a constant $N''_S$ such that we have 
\begin{equation}\label{eq_theta_S_Sigma_1_II_1}
   |d\theta_{S,j}|<\mu_j \cdot \frac{N_S''}{\dej}\overset{(\ref{eq_dfn_mu_1})}{\le}  \frac{N_S''\cdot \ep}{1+|g_{S,j}|+|f|}.
\end{equation}
This shows (\ref{eq_claim_g_f_dtheta}) for the strata of $\Sigma_1$. 

 We now prove (\ref{eq_claim_g_f_dtheta}) in the case of a stratum $S$ of $\Sigma_0$. On $\W_{sc\delta_j}(S)$, because by (\ref{eq_g_f_sigma_1_1}) we have $|g_{S,j}-f|<\ep \cdot \delta_j$, and since  (\ref{eq_theta_S_Sigma_1_I}) holds (again this inequality follows from (\ref{eq_dlambda}) which holds for $S$), we see that $$|d \theta_{S,j}|\cdot |g_{S,j}- f|<N'_S\cdot \ep,$$ yielding (\ref{eq_claim_g_f_dtheta}) on this set.
Moreover, on the complement of this set, because we have $|d\lambda_{S,j}|<\mu_j$, writing the same computation as in (\ref{eq_theta_S_Sigma_1_II_1}), we get (\ref{eq_claim_g_f_dtheta}).

It remains to prove (\ref{eq_claim_g_f_theta}). We 
 start with the case where the stratum $S$ belongs to $\Sigma_1$. By (\ref{eq_delta_ep_approx})  we see that the desired estimate holds on $\V_{r\dej}(\Gamma_{\xib_S})$ (since $\theta_{S,j}\le 1$). On the complement of this set, since $ \theta_{S,j} <\mu_j$, we can write:
 \begin{equation}\label{eq_theta_mu}  \theta_{S,j}\cdot  |dg_{S,j}- df| \le \mu_j \cdot |g_{S,j}- f|_1 \overset{(\ref{eq_dfn_mu_1})}{<}\ep,  \end{equation}                                                                                                                                                                           
yielding (\ref{eq_claim_g_f_theta}) for the strata of $\Sigma_1$.

 We now address the case where $S$ belongs to $\Sigma_0$. For such a stratum, inequality (\ref{eq_g_f_sigma_1_1}) shows that the desired estimate holds on $\W_{sc\dej}(S)$. On the complement of this set, because $\theta_{S,j}<\mu_j$, we see that computation which is carried out in (\ref{eq_theta_mu}) yields the desired estimate. 
\end{proof}

 \begin{subsection}{Approximations on  submanifolds of $\R^n$}\label{sect_plongement}
 The idea is to first establish the desired theorem for a closed definable manifold with boundary of $\R^n$ (Proposition \ref{pro_closed}). We then make use of this result to show that every $\C^\infty$ definable manifold has a closed embedding (Proposition \ref{lem_plongement}).
 
 Given a definable $\C^m$ submanifold $M$ of $\R^n$, $m\ge 2$, there is a definable  neighborhood $U$ of $M$ in $\R^n$ and a definable retraction $r:U\to M$ such that for all $x\in U$, $r(x)$ is the point that realizes the distance from $x$ to $M$. The vector $(x-r(x))$ is then orthogonal to the tangent space to $M$ at $r(x)$ and we say that $(r,U)$ is {\it a tubular neighborhood of $M$}. The mapping $r$ is at least $\C^{m-1}$ and, if $M$ is 
 $\C^\om$, then so is $r$.
 
 We also recall that a definable function $f$ on a set $A\in \D_n$ is said to be $\C^m$ on  $A$, $m\in \N\cup \{\om\}$, if it extends to a $\C^m$ function on an open neighborhood of $A$ in $\R^n$. 

\begin{pro}\label{pro_closed}
 Let $f\in \D^1(M,N)$, where $M$ is a closed $\C^1$ submanifold with boundary  of $\R^n$  and $N$ a $\C^\om$  submanifold of $\R^k$. Given $\ep \in \spl(M)$, there is  $g\in \D^\infty(M,N)$ satisfying $|f-g|_1<\ep$ on $M$.
\end{pro}
\begin{proof}  Since $M$ is closed,  the mapping $f$  extends to a $\C^1$ mapping $\overline{f}:\R^n \to \R^k$. By Proposition \ref{pro_rn}, for every $\delta\in \spn$ there is a $\C^\om$ mapping $\theta:\R^n\to \R^k$ such that  $|\overline{f}-\theta|_1<\delta$ on $\R^n$. Let $(r,U)$ be a tubular neighborhood of $N$ and let $V$ be a closed neighborhood of $M$ in $\overline{f}^{-1}(U)$.
 If $\delta$ is chosen sufficiently small then $\theta(x)\in U$, for all $x\in V$, and we can set for such $x$, $g(x):=r (\theta(x))$.  Given $\ep\in \spl(M)$, if  we choose $\delta$  sufficiently small then $|g-f|_1$ will be smaller  than $\ep$ on $M$.
\end{proof}
 
 The idea to show Proposition \ref{lem_plongement} is to first construct a $\C^2$ embedding that we will approximate by a $\C^\om$ embedding. We thus shall need the following approximation result for manifolds, which is inspired from the techniques developed to study Nash compact manifolds \cite{bcr}.
 
 \begin{pro}\label{pro_plongement}
 Given a $\C^2$ closed definable submanifold $M$ of $\R^n$, there is a $\C^1$ definable diffeomorphism $h:M\to\hat{M}$, with $\hat{M}$  definable closed $\C^\om$ submanifold of $\R^n$. 
 This embedding can be chosen arbitrarily close to the identity in the sense that, given $\ep \in \spl(M)$, we can require that for all $x\in M$ and all unit vector $u\in T_xM$: 
     \begin{equation}\label{eq_h_approx}
      |x-h(x)|< \ep(x) \et |u-d_x h(u)|<\ep(x).
     \end{equation}
 \end{pro}
 \begin{proof}Let $\G_{n,l}$ be the Grassmannian of $l$-dimensional vector subspaces of $\R^n$, where $l:=\dim M$, and  $$E_{n,l}:=\{(v,P)\in \R^n \times \G_{n,l}: v\perp P\}.$$ 
Denote by  $\gamma_{n,l}:E_{n,l}\to \G_{n,l}$ the  vector bundle defined by $ (v,P)\mapsto P$. We will regard $E_{n,l}$ as a submanifold of $\R^k$, for some $k$. Let $(r,U)$ be a tubular neighborhood of $M$ and $w:M\to \G_{n,l}$ the $\C^1$ mapping that assigns to every element $x\in M$ the vector space $T_x M$. Define then a $\C^1$ mapping $f:U\to E_{n,l}$ by $f(x):=(x-r(x),w(r(x)))$. Taking $U$ smaller if necessary, we can assume that it is a closed manifold with boundary.   By Proposition \ref{pro_closed}, for every $\ep \in \spl(U)$, there is a $\C^\om$  mapping $g:U\to \R^k$ satisfying $|f-g|_1<\ep$. The mapping $f$ is transverse to the zero section of the vector bundle $\gamma_{n,l}$ and we have $f^{-1}(\{0_{\R^{n}}\} \times \G_{n,l})=M$.  Consequently, if the approximation $g$ is good enough, this mapping is also transverse to  the zero section of the vector bundle $\gamma_{n,l}$ and the set $$\hat{M}:=g^{-1}(\{0_{\R^{n}}\} \times \G_{n,l})$$ is  a closed $\C^\om$ submanifold of $\R^n$. Moreover, if the approximation is good enough the mapping $r$ induces a $\C^1$ diffeomorphism from $\hat{M}$ onto $M$ isotopic to the identity map. 
 \end{proof}
\begin{pro}\label{lem_plongement}
 Every  $\C^\om$ definable submanifold of $\R^n$ is   $\C^\om$ definably diffeomorphic to a closed $\C^\infty$  definable submanifold of $\R^{n+1}$. 
\end{pro}
\begin{proof}
Let $\Omega$ be a $\C^\om$ submanifold of $\R^n$  and let $\rho:\Omega \to \R $ be the positive continuous function defined by $\rho(x)=d(x,\R^n \setminus\Omega)$. As $\D^2(\Omega)$ is dense in $\D^0(\Omega)$, we can take a $\C^2$ function $\rho':\Omega \to \R $ such that $|\rho-\rho'|<\frac{\rho}{2}$.  The function $\rho'$ is positive and tends to zero as $x$ tends to the frontier of $\Omega$. Let us denote by $M$  the graph of the function $\frac{1}{\rho'}$ and note that this subset is a closed $\C^2$ submanifold of $\R^{n+1}$.   By Proposition \ref{pro_plongement}, for every  $\ep\in \spl(M)$, there is a $\C^\om$ submanifold  $\hat{M}$ of $\R^{n+1}$ and a $\C^1$ diffeomorphism $h:M\to \hat{M}$  for which (\ref{eq_h_approx}) holds.

 Notice that the canonical projection $\Pi:\R^n \times \R\to \R^n$ induces (by restriction) a $\C^2$ diffeomorphism $\pi:M\to \Omega$.  For $\ep$   sufficiently small, the  mapping $\Pi$ therefore also induces (by restriction again) a $\C^\om$ diffeomorphism  $h:\hat{M}\to  \Omega$.
 \end{proof}
 
 \begin{thm}\label{thm_ef_C_1}
   If $M$ is a $\C^\infty$ definable submanifold of $\R^n$ then $\D^\om(M)$ is dense in $\D^1(M)$ for the $\C^1$ Efroymson topology.
 \end{thm}
\begin{proof}By Proposition \ref{lem_plongement}, we can assume that $M$ is closed. The result then follows from Proposition \ref{pro_closed}.
\end{proof}


 \end{subsection}
 \end{section}

\end{document}